\documentclass[12pt]{amsart}
\usepackage{amsmath,mathrsfs}

\usepackage[T1]{fontenc}

\numberwithin{equation}{section}

\newcommand{\R}{{\mathbb R}}

\newcommand{\eps}{\varepsilon}

\renewcommand{\theta}{\vartheta}

\numberwithin{equation}{section}
\newtheorem{theorem}{Theorem}[section]
\newtheorem{proposition}[theorem]{Proposition}
\newtheorem{lemma}[theorem]{Lemma}
\newtheorem{remark}[theorem]{Remark}

\newtheorem{definition}[theorem]{Definition}
\theoremstyle{definition}

\title[Symmetry of minimizers for quasi-linear problems]{On the symmetry of minimizers in \\
constrained quasi-linear problems}

\author{Hichem Hajaiej}
\address{Department of Mathematics
\newline\indent
King Saud University, College of Sciences
\newline\indent
P.O. Box 2455, 11451 Riyadh, Saudi Arabia}
\email{hichem.hajaiej@gmail.com}

\author{Marco Squassina}
\address{Dipartimento di Informatica
\newline\indent
Universit\`a degli Studi di Verona
\newline\indent
C\'a Vignal 2, Strada Le Grazie 15, I-37134 Verona, Italy}
\email{marco.squassina@univr.it}

\thanks{The First author was partially supported by PRIN: {\em Metodi Variazionali e Topologici
nello Studio di Fenomeni non Lineari.}  The author was partially
supported by the Tunisian ARUB project : {\em Analyse Math\'ematique
et Applications 04/UR/15-02. }}

\begin{document}



\begin{abstract}
We provide a simple proof of the radial symmetry of any nonnegative minimizer for a general
class of quasi-linear minimization problems.
\end{abstract}
\maketitle

\section{Introduction and main result}

Let $\Omega$ be either $\R^N$ or a ball $B_R(0)$ centered at the origin in $\R^N$, and define the
functional ${\mathcal E}:W^{1,p}_0(\Omega)\to\R\cup\{+\infty\}$, $1<p<N$, by setting
$$
{\mathcal E}(u)
=\int_\Omega j(u,|Du|)-\int_\Omega F(|x|,u).
$$
Moreover, let ${\mathcal C}\subset W^{1,p}_0(\Omega)$ be a constraint given by
\begin{equation}
    \label{ilvincolo}
{\mathcal C}=\Big\{u\in W^{1,p}_0(\Omega): \,\,\int_\Omega G(u)=1 \Big\}.
\end{equation}
Let us consider the following minimization problem
\begin{equation}
    \label{minpb}
m=\inf_{v\in {\mathcal C}}{\mathcal E}(v), \quad -\infty<m<+\infty.
\end{equation}
A classical problem in the Calculus of Variations is to establish the existence of a solution
to problem~\eqref{minpb} and, in addition, to detect further qualitative properties of the solutions
such as their radial symmetry and monotonicity~\cite{brezis}.
The existence of solutions was extensively investigated, starting from the seminal
contributions of {\sc Lions} \cite{lions1,lions2}.
The main strategies followed to achieve the
latter goal are, on one hand, the moving plane method by {\sc Gidas}, {\sc Ni} and {\sc Nirenberg}~\cite{gnn} and,
on the other, the symmetrization techniques, initiated by {\sc Steiner} and {\sc Schwarz} for sets,
for which we refer the reader to the monographs~\cite{bandle,kawohl,mossino} and the classic~\cite{polya}.
For the semi-linear case $p=2$, $j(s,t)=|t|^2$ and $F=0$, a pioneering study was performed by
{\sc Berestycki} and {\sc Lions} in the celebrated paper~\cite{berlio}.
General radial symmetry results for $j(s,t)=|t|^2$ have been obtained by {\sc Lopes} in~\cite{lopes} via a reflection
argument and a unique continuation principle.
For $j(s,t)=|t|^p$, interesting results have been achieved by {\sc Brock} in~\cite{brock0}
by exploiting rearrangements and strong maximum principle.
For further relevant generalization of these contributions, we refer to the recent work of {\sc Mari\c s}~\cite{maris}.
The works \cite{brock0,lopes,maris} include the case of systems as well and \cite{brock0,maris}
also allow multiple constraints (very general in~\cite{maris}).
The existence of a Schwarz symmetric solution of problem~\eqref{minpb}
under general assumptions on $F$ and $j(u,|Du|)$, allowing growth conditions such as
$$
\alpha_0|Du|^p\leq j(u,|Du|)\leq \alpha(|u|)|Du|^p,\quad\,\, \alpha_0>0,\,\,\,
\text{$\alpha:\R^+\to\R^+$ continuous},
$$
has been recently established by virtue of generalized {\sc P\'olya-Szeg\"o} inequalities \cite{HSq}.
In this paper, focusing on the highly quasi-linear character of our minimization problem, we want to provide,
under rather weak assumptions, a quite simple proof that {\em any} given nonnegative
minimum $v$ of~\eqref{minpb} is radially symmetric and decreasing, after a translation, if the
set of critical points of $v^*$ has null Lebesgue measure. In general,
assuming for instance that $j$ is convex in the gradient and $F$ behaves smoothly,
${\mathcal E}$ is non-smooth unless $j_u=0$ and, depending upon the growth estimates
on $j$, it can be either continuous (if $\alpha$ is bounded from above) or lower semi-continuous.
In turn, quite often, techniques of non-smooth analysis are employed.

Given a nonnegative solution $v$ to~\eqref{minpb},
the idea is to construct a related sequence $(v_n)$ (built up by repeatedly polarizing
$v$) which is weakly convergent to the Schwarz symmetrization $v^*$ of $v$ in
$W^{1,p}_0(\Omega)$. Then, since $(v_n)$ are also solutions to~\eqref{minpb} they satisfy an Euler-Lagrange
equation in a suitable generalized sense (see Section~\ref{nonsmoothsect} and, in particular,
Proposition~\ref{semic-lagrange}) obtained by tools of
subdifferential calculus for nonsmooth functionals developed by {\sc Campa} and {\sc Degiovanni} in~\cite{campad}. This allows, in turn, to
prove the almost everywhere convergence of the gradients $Dv_n$ to $Dv^*$ by applying a powerful result due to {\sc Dal Maso}
and {\sc Murat}~\cite{masmur} to a suitable sequence of Leray-Lions type operators associated with $j(v_n,|Dv_n|)$.
Finally, this leads to the identity $\|Dv\|_{L^p(\Omega)}=\|Dv^*\|_{L^p(\Omega)}$
which provides the desired conclusion that $v$ is nothing but a translation of $v^*$.
We stress that, in proving the main result, we never use any form of the strong maximum principle
or unique continuation principle.
Identity cases for the $p$-Laplacian have been
deeply studied since the first pioneering contributions due to {\sc Friedman} and {\sc McLeod} \cite{fmc}
and to {\sc Brothers} and {\sc Ziemer} \cite{bzim}. For some recent developments, extensions and new simplified proofs, we refer the
reader to the works of {\sc Ferone} and {\sc Volpicelli} (see~\cite{FV,FV1} covering both the case of $\R^N$ and of
a bounded domain).

Beyond the study of minima, for an investigation of radial symmetry of {\em minimax} critical points
for a class of quasi-linear problems on the ball associated with lower
semi-continuous functionals involving $j(u,|Du|)$, we refer
to~\cite{squassi-radial} (see also~\cite{jvsh1} for the case of $C^1$ functionals).
We also refer to the monograph~\cite{sq-monograph} and to the references therein
for a wide range of results on quasi-linear problems obtained via non-smooth critical point theory.
\vskip4pt
\noindent
Throughout the paper, the spaces $L^q(\Omega)$ and $W^{1,p}_0(\Omega)$, for every $p,q\geq 1$, will be
endowed, respectively, both for $\Omega=B_R(0)$ or $\Omega=\R^N$, with the usual norms
$$
\|u\|_{L^q(\Omega)}=\Big(\int_\Omega |u|^q\Big)^{1/q},\quad\,
\|u\|_{W^{1,p}_0(\Omega)}=\Big(\|u\|^p_{L^p(\Omega)}
+\sum_{j=1}^N\|D_ju\|^p_{L^p(\Omega)}\Big)^{1/p}.
$$
\vskip2pt
\noindent
Next we formulate the assumptions under which our main result will hold.

\subsubsection{Assumptions on $j$}
For every $s$ in $\R$, the function ($t\in\R^+$)
\begin{equation}\label{j1}
\text{$\{t\mapsto j(s,t)\}$
is strictly convex and increasing}.
\end{equation}
The functions $j_s$ and $j_t$ and $j_{st}$
denote the derivatives of $j(s,t)$ with
respect to the variables $s$ and $t$ and the mixed derivative
respectively, which exist continuous. We assume that there exist a positive constant $\alpha_0$ and
increasing functions $\alpha,\,\beta,\,\gamma\in C(\R^+,\R^+)$ such that
\begin{align}
    \label{j2}
\alpha_0|\xi|^p\leq  j(s,|\xi|) &\leq \alpha(|s|)|\xi|^p, \qquad\text{for every $s$ in $\R$ and $\xi\in \R^N$,} \\
    \label{j3}
|j_s(s,|\xi|)| &\leq \beta(|s|)|\xi|^p,  \qquad\text{for every $s$ in $\R$ and $\xi\in \R^N$,} \\
\label{j4}
|j_t(s,|\xi|)| &\leq \gamma(|s|)|\xi|^{p-1}, \qquad\text{for every $s$ in $\R$ and $\xi\in \R^N$.}
\end{align}

\subsubsection{Assumptions on $F$}
$F(|x|,s)$ is the primitive with respect to $s$ of a Carath\'eodory function $f(|x|,s)$
with $F(|x|,0)=0$. Denoting $p^*=Np/(N-p)$, we assume that there exist a positive constant $C$ and
a radial function $a\in L^{Np/(N(p-1)+p)}(\Omega)$ such that
\begin{align}
\label{gg1}
|f(|x|,s)| &\leq a(|x|)+C|s|^{p^*-1},\qquad\text{for every $s$ in $\R$ and $x\in \Omega$,} \\
\label{gg5}
f(|x|,s) &\geq f(|y|,s),\qquad\text{for every $s\in\R^+$ and $x,y\in \Omega$ with $|x|\leq |y|$}.
\end{align}

\subsubsection{Assumptions on $G$}
$G(s)$ is the primitive with respect to $s$ of a continuous function $g$
with $G(0)=0$. Moreover, there exists a positive constant $C$ such that
\begin{align}
\label{gg1-G}
|g(s)| &\leq C|s|^{p-1}+C|s|^{p^*-1},\qquad\text{for every $s$ in $\R$,} \\
\label{nondegen}
&\text{if $u\in W^{1,p}_0(\Omega)$ and $u\not\equiv 0$,\, then $g(u)\not\equiv 0$}.
\end{align}
For a given positive $u\in W^{1,p}_0(\Omega)$, we consider the set
$$
C^*:=\{x\in\Omega:|D u^*(x)|=0\}\cap (u^*)^{-1}(0,{\rm esssup}\, u).
$$
\noindent
Under the previous assumptions \eqref{j1}-\eqref{gg1-G}, the main result of the paper is the following

\begin{theorem}
    \label{mainthALL}
Assume that $\Omega$ is either $\R^N$ or a ball $B_R(0)\subset \R^N$ and
let $u\in {\mathcal C}$ be any nonnegative solution to~\eqref{minpb} such that
${\mathcal L}^N(C^*)=0$. Then, after a translation, $u=u^*$.
\end{theorem}

If the problem is not set in a ball or on the whole space, in general minima could fail to be radially symmetric,
even though the domain is invariant under rotations.
For instance, {\sc Esteban}~\cite{est} showed that, if $2<m<2^*$
and $B$ is a closed ball in $\R^N$, then the problem
$$
\min_{u\in H^1(\R^N\setminus{B})}\Big\{\int_{\R^N\setminus{B}}(|D u|^2+|u|^2):
\int_{\R^N\setminus{B}}|u|^{m}=1\Big\}
$$
admits a solution but {\em no} solution is radially symmetric. See also the discussion
by {\sc Kawohl} in \cite[Example 6 and related references]{kawohl-int}
for similar situations of non-symmetric solutions when the problem is defined on an annulus.

Also, as pointed out by {\sc Brothers} and {\sc Ziemer}~\cite[see Section 4]{bzim}
with a counterexample, the condition ${\mathcal L}^N(C^*)=0$
is necessary in order to ensure that $\|Dv\|_{L^p(\Omega)}=\|Dv^*\|_{L^p(\Omega)}$ implies
that $v$ is a translation of $v^*$.

In the particular case where $j(s,t)=|t|^p$, the conclusion of Theorem~\ref{mainthALL} easily follows directly
from identity cases for the $p$-Laplacian operator. In fact, if $u\in {\mathcal C}$ is a nonnegative solution to the minimization problem
and $u^*$ is the Schwarz symmetrization of $u$, then of course $u^*$ belongs
to the constraint ${\mathcal C}$ too (Cavalieri's principle). Moreover, in light of the classical P\'olya-Szeg\"o inequality
and~\eqref{Fsymmm} of Proposition~\ref{concretelem-star}, we have
\begin{equation}
    \label{polya-hardy}
\int_\Omega |D u^*|^p\leq \int_\Omega |D u|^p,\qquad  \int_\Omega F(|x|,u)\leq \int_\Omega F(|x|,u^*).
\end{equation}
Hence,
$$
m\leq {\mathcal E}(u^*)=\int_\Omega |D u^*|^p-\int_\Omega F(|x|,u^*)\leq
\int_\Omega |D u|^p-\int_\Omega F(|x|,u)=m.
$$
In turn, by~\eqref{polya-hardy}, we have both $\|D u^*\|_{L^p(\Omega)}=\|D u\|_{L^p(\Omega)}$
and $\int_\Omega F(|x|,u)=\int_\Omega F(|x|,u^*)$. Then, by~Proposition~\ref{identityc}, if ${\mathcal L}^N(C^*)=0$, there is
a translate of $u^*$ which is equal to $u$.
In the full quasi-linear case, the P\'olya-Szeg\"o inequality (cf.~Proposition~\ref{concretelem-star})
\begin{equation*}
\int_\Omega j(u^*,|D u^*|)\leq \int_\Omega j(u,|D u|)
\end{equation*}
holds as well when $j(u,|D u|)\in L^1(\Omega)$, and the above argument would lead to the identity
\begin{equation}
    \label{idgeneral}
\int_\Omega j(u^*,|D u^*|)=\int_\Omega j(u,|D u|).
\end{equation}
It is not clear (we set it as an open problem) if~\eqref{idgeneral} plus ${\mathcal L}^N(C^*)=0$,
could yield {\em directly} the conclusion that there is
a translate of $u^*$ which is almost everywhere equal to $u$. According to~\cite[Corollary 3.8]{HSq}, this would
hold true knowing in advance that $(u_n)\subset W^{1,p}_0(\Omega)$, $u_n\rightharpoonup v$
weakly and $\int_\Omega j(u_n,|D u_n|)$ convergent to $\int_\Omega j(v,|D v|)$ imply $\|Du_n\|_{L^p(\Omega)}\to \|Dv\|_{L^p(\Omega)}$,
as $n\to\infty$. This is known to be the case for
strictly convex and coercive integrands $j$ which
are merely dependent on the gradient, say $j(s,t)=j_0(t)$, see~\cite{visi}.
In this paper we shall solve the problem indirectly, for minima, by reducing to identity cases for
the $p$-Laplacian operator.
Of course one could derive the radial symmetry information focusing on identity cases of the
nonlinear term, namely from $\int_\Omega F(|x|,u)=\int_\Omega F(|x|,u^*)$. For results
in this direction, under strict monotonicity assumptions of $f$ such as
$$
\text{$f(|x|,s)>f(|y|,s)$, \, for all $s\in\R^+$ and $x,y\in \Omega$ with $|x|<|y|$},
$$
we refer the reader to~\cite[Section 6]{hichem-rse} (see also~\cite{brock0}).

On the basis of the above discussion, the aim of the paper is to
focus the attention of the quasi-linear term in the functional ${\mathcal E}$
(we believe this is somehow more natural since the
strict convexity of $j(s,\cdot)$ is a very common requirement) and show that, for minima,
identity~\eqref{idgeneral} implies, as desired, that $u$ corresponds to a translate of $u^*$.

\begin{remark}\rm
In light of conditions~\eqref{gg1} and~\eqref{gg1-G}, we also have
\begin{align}
\label{gg2}
|F(|x|,s)| &\leq a(|x|)|s|+C|s|^{p^*},\qquad\text{for every $s\in\R$ and $x\in \Omega$,} \\
\label{gg2-G}
|G(s)| &\leq C|s|^p+C|s|^{p^*},\qquad\text{for every $s\in\R$}.
\end{align}
As a possible variant of the growth condition~\eqref{gg1} one could assume that $f:\R^+\times\R\to\R$
with $f(|x|,s)\geq 0$ for every $s\in\R^+$ and $x\in \Omega$ and
\begin{equation}
\label{gg1-bis}
|f(|x|,s)| \leq C|s|^{p-1}+C|s|^{p^*-1},\qquad\text{for every $s$ in $\R$ and $x\in \Omega$,}
\end{equation}
yielding, in turn,
\begin{equation}
\label{gg2-bis}
|F(|x|,s)| \leq C|s|^{p}+C|s|^{p^*},\qquad\text{for every $s$ in $\R$ and $x\in \Omega$.}
\end{equation}
In this case the symmetrization inequality~\eqref{polya-hardy} for $F$
holds as well (use~\cite[Corollary 5.2]{HSt} in place of~\cite[Corollary 5.5]{HSt} in the case $\Omega=\R^N$
and~\cite[Theorem 6.3]{HSt} in place of~\cite[Theorem 6.4]{HSt} in the case $\Omega=B_R(0)$).
It is often the case the $F$ must satisfy
growth conditions which are more restrictive than~\eqref{gg2} or~\eqref{gg2-bis} in order
to have $m>-\infty$. For instance, assume
that $G(s)=|s|^p$, $j(s,t)=|t|^p$ and $F(|x|,s)=|s|^\sigma$.
Then, as a simple scaling argument shows, to guarantee that the minimization problem is well defined
it is necessary to assume that $p<\sigma<p+p^2/N$.
For $p=2$, the value $2+4/N$ is precisely the well-known threshold for orbital stability
of ground states solutions for the nonlinear Schr\"odinger equation.
\end{remark}

\begin{remark}\rm
If $\Omega=\R^N$ and $G(s_0)>0$ at some point $s_0>0$, one can write down a
function $\psi_{s_0}\in W^{1,p}\cap L^\infty_c(\R^N)$ such that $\int_{\R^N} G(\psi_{s_0})=1$
(see \cite[p.325]{berlio}). Moreover, by~\eqref{j2},
    $$
    \int_{\R^N} j(\psi_{s_0},|D\psi_{s_0}|)\leq \alpha(\|\psi_{s_0}\|_{L^\infty})\int_{\R^N} |D\psi_{s_0}|^p<+\infty.
    $$
    Hence $\psi_{s_0}\in {\mathcal C}$ as well as ${\mathcal E}(\psi_{s_0})<+\infty$ (which guarantees $m<+\infty$).
    If $\Omega=B_R(0)$
    and, for instance, $\pi G(s_0) R^2>1$, similarly, one can find $\phi_{s_0}\in C^\infty_c(B_R(0))$
    belonging to ${\mathcal C}$.
\end{remark}

It would be interesting to extend Theorem~\ref{mainthALL}, in a suitable sense, to allow the case of possibly sign-changing solutions,
systems and multiple constraints. The main ingredients of the argument are the facts that the functional decreases
under both polarization and symmetrization, while the constraint remains invariant to them. This can be
achieved for some classes of vectorial problems
putting cooperativity conditions on the nonlinear term $F$ and considering $G$ and $j$
involving a combinations of functions depending only on one single variable, in order to exploit
Cavalieri's principle and Polya-Szeg\"o type inequalities. Notice also that the almost everywhere
convergence of the gradients due to Dal Maso and Murat~\cite{masmur} is valid for systems of PDEs as well.
We leave this issues to further future investigations.

\section{Preliminary facts}

In the section we include some preparatory results.

\subsection{Polarization and Schwarz symmetrization}

For the notions of this section, we refer, for instance, to~\cite{brock}.
A subset $H$ of $\R^N$ is called a polarizer if it is a closed affine half-space
of $\R^N$. Given $x\in\R^N$ and a polarizer $H$, the reflection of $x$ with respect to the boundary of $H$ is
denoted by $x_H$. The polarization of a function $u:\R^N\to\R^+$ by a polarizer $H$
is the function $u^H:\R^N\to\R^+$ defined by
\begin{equation}
 \label{polarizationdef}
u^H(x):=
\begin{cases}
 \max\{u(x),u(x_H)\}, & \text{if $x\in H$} \\
 \min\{u(x),u(x_H)\}, & \text{if $x\in \R^N\setminus H$.} \\
\end{cases}
\end{equation}
The polarization $\Omega^H\subset\R^N$ of a set $\Omega\subset\R^N$ is
defined as the unique set which satisfies $\chi_{\Omega^H}=(\chi_\Omega)^H$,
where $\chi$ denotes the characteristic function.
The polarization $u^H$ of a nonnegative function $u$ defined on $\Omega\subset \R^N$
is the restriction to $\Omega^H$ of the polarization of the extension $\tilde u:\R^N\to\R^+$ of
$u$ by zero outside $\Omega$.
The Schwarz symmetrization of a set $\Omega\subset \R^N$ is the unique open ball
centered at the origin $\Omega^*$ such that ${\mathcal L}^N(\Omega^*)={\mathcal L}^N(\Omega)$, being ${\mathcal L}^N$ the
$N$-dimensional outer Lebesgue measure. If the measure of $\Omega$ is zero we set $\Omega^*=\emptyset$,
while if the measure of $\Omega$ is not finite we put $\Omega^*=\R^N$.
A measurable function $u$ is admissible for the Schwarz symmetrization if it is nonnegative and, for every $\eps>0$,
the Lebesgue measure of $\{u>\eps\}$ is finite. The Schwarz symmetrization
of an admissible function $u:\Omega\to\R^+$ is the unique function $u^*:\Omega^*\to\R^+$ such that,
for all $t\in\R$, it holds $\{u^*>t\}=\{u>t\}^*$. Considering the extension
$\tilde u:\R^N\to\R^+$ of $u$ by zero outside $\Omega$,
$u^*=(\tilde u)^*|_{\Omega^*}$ and $(\tilde u)^*|_{\R^N\setminus \Omega^*}=0$.

\medskip
\noindent
We shall denote by ${\mathcal H}_*$ the set of all half-spaces
corresponding to $(n-1)$-dimensional Euclidean hyperplanes, containing the origin in the interior.
As known, for a domain $\Omega$, it holds $\Omega^*=\Omega$ if and only
if $\Omega^H=\Omega$, for all $H\in {\mathcal H}_*$ (cf.~\cite[Lemma 6.3]{brock}).
We now recall a very useful convergence result (cf.~e.g.~\cite{jvsh2}).

\begin{proposition}
    \label{convsymm}
    Assume that $\Omega$ is either $\R^N$ or a ball $B_R(0)\subset\R^N$.
There exists a sequence of polarizers $(H_m)\subset {\mathcal H}_*$
such that, for any $1\leq p<\infty$ and all $u\in L^p(\Omega)$,
the sequence $u_m=u^{H_1\cdots H_m}$ converges to $u^*$ strongly in $L^p(\Omega)$,
namely $\|u_m-u^*\|_{L^p(\Omega)}\to 0$ as $m\to\infty$.
\end{proposition}

\noindent
We collect some properties of polarizations.
See~\cite[Lemma 2.5]{HSq} and~\cite[Proposition 2.3]{jvsh1} respectively.
Concerning~\cite[Lemma 2.5]{HSq}, it is stated therein on $\R^N$, but it holds on $B_R(0)$ as well after
extending the functions by zero outside it and recalling that $j(\cdot,0)=0$.

\begin{proposition}
    \label{concretelem}
    Let $\Omega$ be either $\R^N$ or a ball $B_R(0)$,
    $u\in W^{1,p}_0(\Omega,\R^+)$ and $H\in {\mathcal H}_*$. Then $u^H\in W^{1,p}_0(\Omega,\R^+)$ and,
    if $j(u,|Du|)\in\ L^1(\Omega)$, then $j(u^H,|Du^H|)\in\ L^1(\Omega)$ and
 \begin{equation}
    \label{fundineq}
        \int_{\Omega}j(u,|Du|)=\int_{\Omega}j(u^H,|D u^H|).
    \end{equation}
    In particular,
    \begin{equation}
        \label{fundineqpart}
            \int_{\Omega}|D u|^p =\int_{\Omega}|D u^H|^p.
        \end{equation}
        Furthermore, $F(|x|,u),F(|x|,u^H),G(u),G(u^H)\in L^1(\Omega)$ and
        $$
        \int_{\Omega} F(|x|,u^H)\geq \int_\Omega F(|x|,u),\quad
        \int_{\Omega} G(u^H)=\int_\Omega G(u),
        $$
provided that conditions~\eqref{gg1},~\eqref{gg5} and~\eqref{gg1-G} hold.
\end{proposition}

\noindent
The next proposition follows by~\cite[Corollary 5.5]{HSt} for the case $\Omega=\R^N$ and~\cite[Theorem 6.4]{HSt} for
the case $\Omega=B_R(0)$ concerning the symmetrization inequality for $F$. Concerning the Cavalieri's principle for $G$,
it follows from~\cite[Theorem 4.4]{HSt} in the case $\Omega=\R^N$ and from~\cite[Theorem 6.2]{HSt} in the case $\Omega=B_R(0)$.
Finally, concerning the generalized P\'olya-Szeg\"o inequality, it follows from~\cite[Corollary 3.3]{HSq}.

\begin{proposition}
    \label{concretelem-star}
    Let $\Omega$ be either $\R^N$ or a ball $B_R(0)$ and let
    $u\in W^{1,p}_0(\Omega,\R^+)$. Then $u^*\in W^{1,p}_0(\Omega,\R^+)$ and,
    if $j(u,|Du|)\in\ L^1(\Omega)$, then $j(u^*,|Du^*|)\in\ L^1(\Omega)$ and
 \begin{equation}
    \label{fundineq-star}
        \int_{\Omega}j(u^*,|Du^*|)\leq \int_{\Omega}j(u,|D u|).
    \end{equation}
    In particular,
    $$
    \int_{\Omega} |Du^*|^p\leq \int_{\Omega} |Du|^p.
    $$
    Furthermore, $F(|x|,u),F(|x|,u^*),G(u),G(u^*)\in L^1(\Omega)$ and
        \begin{equation}
            \label{Fsymmm}
        \int_{\Omega} F(|x|,u^*)\geq \int_\Omega F(|x|,u),\quad
        \int_{\Omega} G(u^*)=\int_\Omega G(u),
    \end{equation}
provided that conditions~\eqref{gg1},~\eqref{gg5} and~\eqref{gg1-G} hold.
\end{proposition}

\noindent
The next result comes from~\cite{FV} for $\Omega$ bounded
and~\cite{FV1} for $\Omega=\R^N$.

\begin{proposition}
    \label{identityc}
Assume that $\Omega$ is an open, bounded subset of $\R^N$ and let
$u\in W^{1,p}_0(\Omega)$ be a nonnegative function, $1<p<\infty$, such that
$$
{\mathcal L}^N(C^*)=0,\quad
C^*:=\{x\in\Omega^*:|D u^*(x)|=0\}\cap (u^*)^{-1}(0,{\rm esssup}\, u).
$$
Then, if
$$
\|D u^*\|_{L^p(\Omega^*)}=\|D u\|_{L^p(\Omega)},
$$
the domain $\Omega$ is equivalent to a ball and $u=u^*$ a.e.~in $\Omega$, up to a translation.
Moreover, the same conclusion holds for $\Omega=\R^N$.
\end{proposition}

\subsection{Solutions to the Euler-Lagrange equation}
\label{nonsmoothsect}

For any $u\in W^{1,p}_0(\Omega)$, define
\begin{equation}\label{defvu}
V_u=\big\{v\in W^{1,p}_0(\Omega)\cap L^\infty(\Omega):\,
u\in L^{\infty}(\{x\in\Omega:\,v(x)\not=0\})\big\}.
\end{equation}
The vector space $V_u$ was firstly introduced by {\sc Degiovanni} and {\sc Zani} in~\cite{degzan} in the case $p=2$.
In~\cite{degzan} it is also proved that $V_u$ with $p=2$ is dense in $W^{1,2}_0(\Omega)$. This fact
extends with the same proof to the general case of any $p\neq 2$.
Let $J:W^{1,p}_0(\Omega)\to\R\cup\{+\infty\}$ be the functional
$$
J(u)=\int_\Omega j(u,|Du|).
$$
The following fact can be easily checked. It shows that $V_u$ is a good test space
to differentiate non-smooth functionals of calculus of variations satisfying
suitable growth conditions.

\begin{proposition}
    \label{derivate}
Assume conditions~\eqref{j2}, \eqref{j3} and \eqref{j4}.
Then, for every $u\in W^{1,p}_0(\Omega)$ with $J(u)<+\infty$
and every $v\in V_u$ we have
$$
j_s(u,|Du|)v\in L^1(\Omega),\qquad
j_t(u,|Du|)\frac{Du}{|Du|}\cdot D v\in L^1(\Omega),
$$
with the agreement that $j_t(u,|Du|)\frac{Du}{|Du|}=0$ when $|Du|=0$ (in view of~\eqref{j4}). Moreover,
the function $\{t\mapsto J(u+tv)\}$ is of class $C^1$ and
$$
J'(u)(v)=\int_{\Omega}j_t(u,|Du|)\frac{Du}{|Du|}\cdot Dv
+\int_{\Omega}j_s(u,|Du|)v.
$$
\end{proposition}

\noindent
We recall Definitions 4.3 and 5.5 from~\cite{campad}, respectively, adapted to our concrete framework.

\begin{definition}
Let $u\in W^{1,p}_0(\Omega)$ with $J(u)<+\infty$. For every $v\in W^{1,p}_0(\Omega)$
and $\eps>0$ we define $J_\eps^0(u;v)$ to be the infimum of the $r$'s in $\R$ such that
there exist $\delta>0$ and a continuous function
$$
{\mathcal V}: B_\delta (u,J(u))\cap {\rm epi}(J)\times ]0,\delta]\to B_\eps(v),
$$
which satisfies
\begin{equation*}
J(\xi+t{\mathcal V}((\xi,\mu),t))\leq \mu+rt,
\end{equation*}
whenever $(\xi,\mu)\in B_\delta(u,J(u))\cap {\rm epi}(J)$
and $t\in ]0,\delta]$. Finally, we set
$$
J^0(u;v):=\sup_{\eps>0}J_\eps^0(u;v).
$$
\end{definition}

\begin{definition}
    \label{defbar}
Let $u\in W^{1,p}_0(\Omega)$ with $J(u)<+\infty$. For every $v\in W^{1,p}_0(\Omega)$
and $\eps>0$ we define $\bar J_\eps^0(u;v)$ to be the infimum of the $r$'s in $\R$ such that
there exist $\delta>0$ and a continuous function
$$
{\mathcal H}: B_\delta (u,J(u))\cap {\rm epi}(J)\times [0,\delta]\to W^{1,p}_0(\Omega),
$$
which satisfies ${\mathcal H}((\xi,\mu),0)=\xi$,
\begin{equation}
\label{norma}
\Big\|\frac{{\mathcal H}((\xi,\mu),t_1)-{\mathcal H}((\xi,\mu),t_2)}{t_1-t_2}-v\Big\|_{W^{1,p}(\Omega)}<\eps,
\end{equation}
and $J({\mathcal H}((\xi,\mu),t))\leq \mu+rt$, whenever $(\xi,\mu)\in B_\delta(u,J(u))\cap {\rm epi}(J)$
and $t,t_1,t_2\in [0,\delta]$ with $t_1\neq t_2$. Finally, we set
$$
\bar J^0(u;v):=\sup_{\eps>0}\bar J_\eps^0(u;v).
$$
\end{definition}

\noindent
As remarked in~\cite[cf.~p.1037]{campad} it always holds $J^0(u;v)\leq \bar J^0(u;v)$.
Recalling that $\partial J(u)$ is the subdifferential introduced in \cite[Definition 4.1]{campad}, we have the following

\begin{lemma}
    \label{subdcaract}
    Assume conditions~\eqref{j2}, \eqref{j3} and \eqref{j4}.
    Let $u\in W^{1,p}_0(\Omega)$ with $J(u)<+\infty$. Then, the following facts hold:
    \begin{itemize}
        \item[{\rm (i)}] for every $v\in V_u$, we have
        $$
        J^0(u;v)\leq \bar J^0(u;v)\leq \int_\Omega j_t(u,|Du|)\frac{Du}{|Du|}\cdot Dv +
        \int_\Omega j_s(u,|Du|)v.
        $$
        \item[{\rm (ii)}] if $\partial J(u)\neq\emptyset$, then $\partial J(u)=\{\alpha\}$ with $\alpha\in W^{-1,p'}(\Omega)$ and
        \begin{equation*}
            \int_\Omega j_t(u,|Du|)\frac{Du}{|Du|}\cdot D v +
            \int_\Omega j_s(u,|Du|)v=\langle \alpha,v\rangle,
        \end{equation*}
        for all $v\in V_u$.
    \end{itemize}
\end{lemma}
\begin{proof}
Let $\eta>0$ with $J(u)<\eta$. Moreover, let
$v\in V_u$ and $\eps>0$. Take now $r\in\R$ with
\begin{equation}
    \label{step1}
\int_\Omega j_t(u,|Du|)\frac{Du}{|Du|}\cdot D v +
\int_\Omega j_s(u,|Du|)v<r.
\end{equation}
Let $H$ be a $C^\infty(\R)$ function such that
\begin{equation}
    \label{defh}
H(s)=1\,\,\, \text{on $[-1,1]$},
\quad H(s)=0\,\,\, \text{outside $[-2,2]$},
\quad |H'(s)|\leq 2\,\,\, \text{on $\R$}.
\end{equation}
Then, there exists $k_0\geq 1$ such that
\begin{equation}
    \label{prima2}
\Big\|H(\frac{u}{k_0})v-v\Big\|_{W^{1,p}(\Omega)}<\eps,
\end{equation}
and
\begin{equation}
    \label{step2}
\int_\Omega j_t(u,|Du|)\frac{Du}{|Du|}\cdot D \Big(H(\frac{u}{k_0})v\Big) +
\int_\Omega j_s(u,|Du|)H(\frac{u}{k_0})v<r.
\end{equation}
In fact, setting $v_k=H(u/k)v$, we have $v_k\in V_u$ for every
$k\geq 1$ and $v_k$ converges to $v$ in $W^{1,p}_0(\Omega)$, yielding inequality~\eqref{prima2},
for $k$ large enough. By Proposition~\ref{derivate}, we can consider $J'(u)(v_k)$ for all $k\geq 1$
and, as $k$ goes to infinity, for a.e.\ $x\in\Omega$, we have
\begin{gather*}
j_s(u(x),|D u(x)|)v_k(x)\to
j_s(u(x),|D u(x)|)v(x),\qquad
\\
j_t(u(x),|D u(x)|)\frac{Du(x)}{|Du(x)|}\cdot Dv_k (x)\to
j_t(u(x),|D u(x)|)\frac{Du(x)}{|Du(x)|}\cdot Dv(x),
\end{gather*}
as well as
\begin{gather*}
\big|j_s(u,|Du|)v_k\big|\leq |j_s(u,|D u|)| |v|,
\\
\big|j_t(u,|Du|)\frac{Du}{|Du|}\cdot Dv_k\big|\leq
\left|j_t(u,|Du|)||Dv\right|
+2|v||j_t(u,|Du|)||D u|.
\end{gather*}
Since $v\in V_u$ and by the growth estimates~\eqref{j3}-\eqref{j4}, by dominated convergence we get
\begin{gather*}
\lim\limits_{k\to \infty}\int_\Omega j_s(u,|Du|)v_k\,
=\int_\Omega j_s(u,|D u|)v\, ,
\\
\lim\limits_{k\to \infty}\int_{\Omega}j_t(u,|D
u|)\frac{Du}{|Du|}\cdot Dv_k =\int_{\Omega}j_t(u,|Du|)\frac{Du}{|Du|}\cdot Dv,
\end{gather*}
which, together with~\eqref{step1}, yields~\eqref{step2}.
Let us now prove that there exists $\delta_1>0$ such that
\begin{equation}
    \label{prima3-mod}
\Big\|H(\frac{z}{k_0})v-v\Big\|_{W^{1,p}(\Omega)}<\eps,
\end{equation}
as well as
\begin{align}
    \label{step3-mod}
&   \int_\Omega j_t(z+\theta H(\frac{z}{k_0})v,|Dz+\theta D\big(H(\frac{z}{k_0})v\big)|)\frac{Dz+\theta D\big(H(\frac{z}{k_0})v\big)}{|Dz+\theta D\big(H(\frac{z}{k_0})v\big)|}\cdot D \Big(H(\frac{z}{k_0})v\Big) \\
&   +
\int_\Omega j_s(z+\theta H(\frac{z}{k_0})v,|Dz+\theta D\big(H(\frac{z}{k_0})v\big)|)H(\frac{z}{k_0})v<r,   \notag
\end{align}
for all $z\in B(u,\delta_1)\cap J^\eta$ and $\theta\in [0,\delta_1)$. Indeed, take
$u_n\in J^\eta$ such that $u_n \to u$ strongly in $W^{1,p}_0(\Omega)$, $\theta_n\to 0$ as $n\to\infty$ and consider
$v_n=H(u_n/k_0)v\in V_{u_n}$. It follows that $v_n$ converges to $H(u/k_0)v$
strongly in $W^{1,p}_0(\Omega)$, so that~\eqref{prima3-mod} follows by~\eqref{prima2}. Now,
for a.e.~$x\in\Omega$, we have
\begin{gather*}
j_s(u_n(x)+\theta_n v_n(x),|Du_n(x)+\theta_n Dv_n(x)|)v_n(x) \to j_s(u(x),|Du(x)|)H\big(\frac{u(x)}{k_0}\big)v(x)\\
j_t(u_n(x)+\theta_n v_n(x),|Du_n(x)+\theta_n Dv_n(x)|)\frac{Du_n(x)+\theta_n Dv_n(x)}{|Du_n(x)+\theta_n Dv_n(x)|}\cdot Dv_n(x) \\
\to
j_t(u(x),|Du(x)|)\frac{Du(x)}{|Du(x)|}\cdot D\big(H\big(\frac{u}{k_0}\big)v\big)(x).
\end{gather*}
Moreover, we have
\begin{gather*}
\left|j_s(u_n+\theta_n v_n,|Du_n+\theta_n Dv_n|)v_n\right|  \\
\leq  2^{p-1}\beta(2k_0+\|v\|_{L^\infty(\Omega)})
\|v\|_{L^\infty(\Omega)}(|Du_n|^p+|Dv_n|^p),   \\
\big|j_t(u_n+\theta_n v_n,|Du_n+\theta_n Dv_n|)\frac{Du_n+\theta_n Dv_n}{|Du_n+\theta_n Dv_n|}\cdot Dv_n\big| \\
\leq 2^{p-1}\gamma(2k_0+\|v\|_{L^\infty(\Omega)})(\big |D u_n|^{p-1}|Dv_n|+|Dv_n|^p\big).
\end{gather*}
Then, by dominated convergence we obtain
\begin{gather*}
\lim_{n\to \infty}\int_\Omega j_s(u_n+\theta_n v_n,|Du_n+\theta_n Dv_n|)v_n \\
=\int_\Omega j_s(u,|Du|)H\big(\frac{u}{k_0}\big)v\,,\\
\lim_{n\to \infty}\int_{\Omega}j_t(u_n+\theta_n v_n,|Du_n+\theta_n Dv_n|)\frac{Du_n+\theta_n Dv_n}{|Du_n+\theta_n Dv_n|}\cdot Dv_n \\
=\int_{\Omega}j_t(u,|Du|)\frac{Du}{|Du|}\cdot D\Big( H\Big(
\frac{u}{k_0}\Big)v\Big),
\end{gather*}
which, in light of~\eqref{step2}, proves~\eqref{step3-mod}. Then, taking into account that
$\{t\mapsto J(z+t H(\frac{z}{k_0})v)\}$ is of class $C^1$, Lagrange theorem and~\eqref{step3-mod} yield,
for some $\theta\in [0,t]$,
\begin{align*}
& J(z+t H(\frac{z}{k_0})v)-J(z)=t J'(z+\theta H(\frac{z}{k_0})v) \Big(H(\frac{z}{k_0})v\Big) \\
&=t\int_\Omega \Big[ j_t(z+\theta  H(\frac{z}{k_0})v,|Dz+\theta D\big(H(\frac{z}{k_0})v\big)|)\frac{Dz+\theta  D\big(H(\frac{z}{k_0})v\big)}{|Dz+\theta D\big(H(\frac{z}{k_0})v\big)|}\cdot D \big(H(\frac{z}{k_0})v\big) \\
    &   + j_s(z+\theta  H(\frac{z}{k_0})v,|Dz+\theta  D\big(H(\frac{z}{k_0})v\big)|)H(\frac{z}{k_0})v\Big]dx \leq rt    ,
\end{align*}
for all $z\in B(u,\delta_1)\cap J^\eta$ and $t\in [0,\delta_1)$.
Let now $\delta\in (0,\delta_1]$
with $J(u)+\delta<\eta$, and define the continuous function
${\mathcal H}: B_\delta (u,J(u))\cap {\rm epi}(J)\times [0,\delta]\to W^{1,p}_0(\Omega)$
by setting
$$
{\mathcal H}((z,\mu),t)=z+t H(\frac{z}{k_0})v.
$$
Then, by direct computation, condition~\eqref{norma} in Definition~\ref{defbar}
is satisfied by~\eqref{prima3-mod}.
Notice that, for all $((z,\mu),t)\in B_\delta (u,J(u))\cap {\rm epi}(J)\times [0,\delta]$, we have
$z\in B(u,\delta_1)\cap J^\eta$ and $t\in [0,\delta_1)$. Hence, by the above inequality, we have
$$
J({\mathcal H}((z,\mu),t))\leq J(z) +rt\leq \mu+rt.
$$
whenever $(z,\mu)\in B_\delta(u,J(u))\cap {\rm epi}(J)$ and $t\in [0,\delta]$. Then, according to
Definition~\ref{defbar}, we can conclude that $\bar J_\eps^0(u;v)\leq r$.
By the arbitrariness of $r$, it follows that
$$
\bar J_\eps^0(u;v)\leq \int_\Omega j_t(u,|Du|)\frac{Du}{|Du|}\cdot Dv +
\int_\Omega j_s(u,|Du|)v.
$$
Hence, by the arbitrariness of $\eps$, we get
\begin{equation}
    \label{conclpreli}
\bar J^0(u;v)\leq \int_\Omega j_t(u,|Du|)\frac{Du}{|Du|}\cdot Dv +
\int_\Omega j_s(u,|Du|)v,
\end{equation}
for all $v\in V_u$, concluding the proof of assertion (i). Concerning (ii), if $\alpha\in \partial J(u)\subset W^{-1,p'}(\Omega)$,
by (i) it follows (recall~\cite[Corollary 4.7(i)]{campad} for the first inequality below)
$$
\langle\alpha,v\rangle\leq J^0(u;v)\leq \int_\Omega j_t(u,|Du|)\frac{Du}{|Du|}\cdot Dv +
\int_\Omega j_s(u,|Du|)v,
$$
for all $v\in V_u$. Since we can exchange $v$ with $-v$ we get
$$
\langle\alpha,v\rangle=\int_\Omega j_t(u,|Du|)\frac{Du}{|Du|}\cdot Dv +
\int_\Omega j_s(u,|Du|)v,
$$
for all $v\in V_u$. By density of $V_u$ in $W^{1,p}_0(\Omega)$, $\partial J(u)=\{\alpha\}$.
This concludes the proof.
\end{proof}

Finally, we have the following

\begin{proposition}
    \label{semic-lagrange}
    Assume~\eqref{j2}, \eqref{j3}, \eqref{j4}, \eqref{gg1} and \eqref{gg1-G}. Let ${\mathcal C}$
    be the constraint~\eqref{ilvincolo} and let ${\mathcal E}:W^{1,p}_0(\Omega)\to\R\cup\{+\infty\}$ be the functional
    $$
    {\mathcal E}(u)=\int_\Omega j(u,|Du|)-\int_\Omega F(|x|,u).
    $$
    Then the functional ${\mathcal E}^*:W^{1,p}_0(\Omega)\to\R\cup\{+\infty\}$ defined by
    $$
    {\mathcal E}^*(u)
    =
    \begin{cases}
    {\mathcal E}(u) & \text{for $u\in {\mathcal C}$}, \\
    +\infty & \text{for $u\not \in {\mathcal C}$},
    \end{cases}
    $$
     is lower semi-continuous on $W^{1,p}_0(\Omega)$. Moreover, for every solution
$u\in {\mathcal C}$ to problem~\eqref{minpb}
there exists a Lagrange multiplier $\lambda\in\R$ such that
\begin{equation*}
    \int_\Omega j_t(u,|Du|)\frac{Du}{|Du|}\cdot D \varphi +
    \int_\Omega j_s(u,|Du|)\varphi -\int_\Omega f(|x|,u)\varphi =\lambda \int_\Omega g(u)\varphi,
\end{equation*}
for all $\varphi\in V_u$.
\end{proposition}
\begin{proof}
    It is readily seen that ${\mathcal E}^*$ is lower semi-continuous,
    by conditions~\eqref{gg2} and~\eqref{gg2-G}.
Let $u\in {\mathcal C}$ be any solution to problem~\eqref{minpb}
(it is $J(u)<+\infty$, since ${\mathcal E}(u)=m<+\infty$).
Notice that ${\mathcal E}^*={\mathcal E}+{\rm I}_{{\mathcal C}}$,
being ${\rm I}_{{\mathcal C}}$ the indicator function of ${\mathcal C}$, ${\rm I}_{{\mathcal C}}(v)=0$ if $v\in {\mathcal C}$
and ${\rm I}_{{\mathcal C}}(v)=+\infty$ if $v\not \in {\mathcal C}$.
If $|d{\mathcal E}^*|(u)$ denotes the weak slope of ${\mathcal E}^*$
(see~\cite[Definition 2.1]{campad}), it follows $|d{\mathcal E}^*|(u)=0$, and then
$0\in\partial {\mathcal E}^*(u)$, by \cite[Theorem 4.13(iii)]{campad}.
Setting
$$
V(u)=-\int_\Omega F(|x|,u)dx,\quad\,
W(u)=\int_\Omega G(u)dx,
$$
$V,W$ are $C^1$ functionals, in light of~\eqref{gg1} and~\eqref{gg1-G}, and
$\partial V(u)=\{-f(|x|,u)\}$, $\partial W(u)=\{g(u)\}$. Moreover, by~\eqref{nondegen},
we can find $\hat v \in W^{1,p}_0(\Omega)$ such that
$\int_\Omega g(u)\hat v>0$ and, by density of $V_u$ in $W^{1,p}_0(\Omega)$, $v_+ \in V_u$ with
$\int_\Omega g(u)v_+>0$. Taking into account Proposition~\ref{derivate} and
conclusion (i) of Lemma~\ref{subdcaract}, $\bar J^0(u;v_+)<+\infty$
and $\bar J^0(u;-v_+)<+\infty$. Therefore, we are allowed to apply~\cite[Corollary 5.9(ii)]{campad}, yielding
$$
0\in  \partial {\mathcal E}(u)+\R \partial W(u)=\partial J(u)+\partial V(u)+\R\partial W(u),
$$
where the equality is justified by \cite[Corollary 5.3(ii)]{campad}, since $V$ is a $C^1$ functional.
Finally, since $\partial J(u)\neq\emptyset$, assertion (ii) of Lemma~\ref{subdcaract}
allows to conclude the proof.
\end{proof}

\medskip

\section{Proof of Theorem~\ref{mainthALL}}
\noindent
Let $u\in W^{1,p}_0(\Omega)$ be a given nonnegative solution to the minimum problem~\eqref{minpb}, namely
$$
j(u,|Du|)\in L^1(\Omega),\quad
\int_\Omega G(u)=1,\quad
m=\int_\Omega j(u,|Du|)-\int_\Omega F(|x|,u).
$$
We shall divide the proof into four steps.
\vskip4pt
\noindent
{\bf Step I (existence of approximating minimizers).}
In light of Proposition~\ref{convsymm}, we can find
a sequence $(u_n)\subset W^{1,p}_0(\Omega)$ of polarizations of $u$, namely $u_n=u^{H_1\cdots H_n}$, such that
$u_n\to u^*$ strongly in $L^p(\Omega)$, as $n\to\infty$. Furthermore, we learn
from~\eqref{fundineqpart} of Proposition~\ref{concretelem} that
    \begin{equation}
        \label{gradientpp}
    \|D u_n\|_{L^p(\Omega)}=\|D u\|_{L^p(\Omega)},\quad\text{for all $n\geq 1$}.
\end{equation}
In turn, the sequence $(u_n)$ is bounded in $W^{1,p}_0(\Omega)$ and, up to a subsequence,
it converges weakly to $u^*$ in $W^{1,p}_0(\Omega)$
(since $u_n\to u^*$ in $L^p(\Omega)$). Notice also that, again by virtue of Proposition~\ref{concretelem},
it follows that $j(u_n,|Du_n|)\in L^1(\Omega)$ for all $n\geq 1$ and
    \begin{equation}
        \label{energyconv}
    \int_\Omega j(u_n,|Du_n|)=\int_\Omega j(u,|Du|),\quad\text{for all $n\geq 1$},
\end{equation}
    as well as
    $$
    \int_\Omega F(|x|,u_n)\geq \int_\Omega F(|x|,u),\quad
    \int_\Omega G(u_n)=\int_\Omega G(u)=1,\quad\text{for all $n\geq 1$}.
    $$
    In particular $(u_n)$ is a sequence of minimizers for problem~\eqref{minpb}, since $(u_n)\subset {\mathcal C}$ and
    $$
    m\leq \int_\Omega j(u_n,|Du_n|)-\int_\Omega F(|x|,u_n)\leq \int_\Omega j(u,|Du|)-\int_\Omega F(|x|,u)=m.
    $$
    Furthermore, in light of Proposition~\ref{concretelem-star}, we have
    $\int_\Omega G(u^*)=\int_\Omega G(u)=1$ and
    $$
    \int_\Omega F(|x|,u^*)\geq \int_\Omega F(|x|,u),\quad
    \int_\Omega j(u^*,|D u^*|)\leq \int_\Omega j(u,|D u|),
    $$
    so that we obtain
    $$
    m\leq \int_\Omega j(u^*,|Du^*|)-\int_\Omega F(|x|,u^*)\leq \int_\Omega j(u,|Du|)-\int_\Omega F(|x|,u)=m,
    $$
    yielding that $u^*$ is a minimizer for~\eqref{minpb} too, and
    \begin{equation*}
    \int_\Omega j(u^*,|D u^*|)=\int_\Omega j(u,|D u|).
    \end{equation*}
    In conclusion, by~\eqref{energyconv}, we get
    \begin{equation}
            \label{energyconv-final}
        \int_\Omega j(u_n,|Du_n|)=\int_\Omega j(u^*,|Du^*|),\quad\text{for all $n\geq 1$}.
    \end{equation}
    By Proposition~\ref{semic-lagrange}, there exists a
    sequence $(\lambda_n)\subset\R$ of Lagrange multipliers such that
\begin{equation}
    \label{equEL}
    \int_\Omega j_t(u_n,|Du_n|)\frac{Du_n}{|Du_n|}\cdot D \varphi +
    \int_\Omega j_s(u_n,|Du_n|)\varphi -\int_\Omega f(|x|,u_n)\varphi =\lambda_n \int_\Omega g(u_n)\varphi ,
\end{equation}
    for all $n\geq 1$ and any $\varphi\in V_{u_n}$.
    \vskip6pt
    \noindent
    {\bf Step II (boundedness of $\boldsymbol{\lambda_n}$).}
    We claim that $(\lambda_n)$ is bounded in $\R$.
        To prove this, observe first that there exist $\hat v\in C^\infty_c(\Omega)$ and $\hat h\geq 1$ such that
        \begin{equation}
            \label{lambdalim}
        \lim_{n}\int_{\Omega} g(u_n)H\Big(\frac{u_n}{\hat h}\Big)\hat v\not= 0,
    \end{equation}
being $H$ the cut-off function defined in~\eqref{defh}.
If this was not the case, for all $v\in C^\infty_c(\Omega)$ and any $h\geq 1$, we would find
        $$
        \int_{\Omega} g(u^*)H\Big(\frac{u^*}{h}\Big)v=
        \lim_{n}\int_{\Omega} g(u_n)H\Big(\frac{u_n}{h}\Big)v=0,
        $$
        by dominated convergence. By the arbitrariness of $h\geq 1$ and dominated convergence, we get
        $$
        \int_{\Omega} g(u^*)v=0,\quad\text{for all $v\in C^\infty_c(\Omega)$},
        $$
        so that $g(u^*)=0$ a.e.~in $\Omega$. This is a contradiction, as $\int_\Omega G(u^*)=\int_\Omega G(u)=1$
        implies that $u^*\not \equiv 0$ which, by assumption~\eqref{nondegen}, yields $g(u^*)\not \equiv 0$.
        Observe now that, for every $v\in C^\infty_c(\Omega)$ and any $h\geq 1$, the function
        $H(\frac{u_n}{h})v$ belongs to $V_{u_n}$ and thus it
        is an admissible test function for~\eqref{equEL}.
Therefore, if $\hat v\in C^\infty_c(\Omega)$ and $\hat h\geq 1$ are as in formula~\eqref{lambdalim},
inserting $\varphi=H(\frac{u_n}{\hat h})\hat v$ into~\eqref{equEL}, we reach the identity
        \begin{equation}
            \label{lambdaid}
        \lambda_n\int_{\Omega} g(u_n)H\Big(\frac{u_n}{\hat h}\Big)\hat v=\sum_{i=1}^3 I_i^n,
    \end{equation}
        where, denoted by $K$ the support of $\hat v$, we have set
        \begin{align*}
        I_1^n& :=\int_{K} H\Big(\frac{u_n}{\hat h}\Big)j_t(u_n,|D
        u_n|)\frac{Du_n}{|Du_n|}\cdot D \hat v,   \\
        I_2^n &:=\int_{K} \left[H\Big(\frac{u_n}{\hat h}\Big)j_s(u_n,|D u_n|)+
        H'\Big(\frac{u_n}{\hat h}\Big)j_t(u_n,|D u_n|)\frac{|D u_n|}{\hat h}\right]\hat v, \\
        I_3^n &:=-\int_{K} f(|x|,u_n)H\Big(\frac{u_n}{\hat h}\Big)\hat v.
    \end{align*}
    In turn, taking into account the growths~\eqref{j3},~\eqref{j4} and~\eqref{gg1}, it follows
    \begin{align*}
    |I_1^n| &\leq C\int_K |D u_n|^{p-1}|D \hat v|\leq C\Big(\int_\Omega |D u_n|^p\Big)^{\frac{p-1}{p}}\leq C, \\
    |I_2^n| &\leq C\int_K |D u_n|^p|\hat v|\leq C \int_\Omega |D u_n|^p\leq C, \\
    |I_3^n| &\leq \int_K (a(|x|)+C)|\hat v|\leq C,
    \end{align*}
    for some constant $C=C(\hat h)$, changing from one line to the next and independent of $n$.
    Then the claim follows by combining~\eqref{lambdalim} and~\eqref{lambdaid} and $(\lambda_n)$
    admits a convergent subsequence.
    \vskip3pt
    \noindent
    {\bf Step III (pointwise convergence).}
    In this step we prove that, up to a subsequence,
    \begin{equation}\label{convun2}
    D u_n(x)\to D u^*(x),\qquad \text{for a.e.\ $x\in\Omega$}.
    \end{equation}
    Let $\Omega_0$ be a fixed bounded subdomain of $\Omega$ (let $\Omega_0=\Omega$ if $\Omega$ is a ball). We already know that
    \begin{equation}\label{weakustarbdd}
    u_n\rightharpoonup u^*\quad\text{weakly in $W^{1,p}(\Omega_0)$}.
\end{equation}
    As we have already noticed,
    for all $h\geq 1$ and $v\in C^\infty_c(\Omega_0)$,
    the function $H(\frac{u_n}{h})v$ belongs to the space $V_{u_n}$.
    Therefore, inserting it into~\eqref{equEL}, we reach
    \begin{gather*}
    \int_{\Omega_0} H\Big(\frac{u_n}{h}\Big)j_t(u_n,|D
    u_n|)\frac{Du_n}{|Du_n|}\cdot D v\\
    =-\int_{\Omega_0} \left[H\Big(\frac{u_n}{h}\Big)j_s(u_n,|D u_n|)+
    H'\Big(\frac{u_n}{h}\Big)j_t(u_n,|D u_n|)\frac{|D u_n|}{h}\right]v\\
    +\int_{\Omega_0} f(|x|,u_n)H\Big(\frac{u_n}{h}\Big)v+\lambda_n\int_{\Omega_0} g(u_n)H\Big(\frac{u_n}{h}\Big)v,
    \end{gather*}
    for all $v\in C^\infty_c(\Omega_0)$. This equality can be read as
\begin{equation*}
    \int_{\Omega_0} b_n(x,Du_n) \cdot D v\\
    =\langle \Phi_n,v \rangle+\langle \mu_n,v\rangle,\quad \forall v\in C^\infty_c(\Omega_0),
\end{equation*}
where we have set
\begin{align}
    b_n(x,\xi) &:=H\Big(\frac{u_n(x)}{h}\Big)j_t(u_n(x),|\xi|)\frac{\xi}{|\xi|},\quad\text{for a.e.~$x\in\Omega_0$ and all $\xi\in\R^N$,}  \notag \\
    \label{misradon}
\langle \mu_n,v\rangle &:=
    -\int_{\Omega_0}\left[H'\Big(\frac{u_n}{h}\Big)j_t(u_n,|D u_n|)\frac{|D
    u_n|}{h}+
    H\Big(\frac{u_n}{h}\Big)j_s(u_n,|D u_n|)\right]v,   \\
\langle \Phi_n,v \rangle &:=
    \int_{\Omega_0} f(|x|,u_n)H\Big(\frac{u_n}{h}\Big)v+\lambda_n\int_{\Omega_0} g(u_n)H\Big(\frac{u_n}{h}\Big)v, \notag
\end{align}
for every $v\in C^\infty_c(\Omega_0)$. Set also
\begin{align*}
b(x,\xi) &:=H\Big(\frac{u^*(x)}{h}\Big)j_t(u^*(x),|\xi|)\frac{\xi}{|\xi|},\quad\text{for a.e.~$x\in\Omega_0$ and all $\xi\in\R^N$},  \\
\langle \Phi,v \rangle &:=
    \int_{\Omega_0} f(|x|,u^*)H\Big(\frac{u^*}{h}\Big)v+\lambda\int_{\Omega_0} g(u^*)H\Big(\frac{u^*}{h}\Big)v,
    \quad \forall v\in C^\infty_c(\Omega_0),
\end{align*}
where $\lambda$ denotes the limit of $(\lambda_n)$, according to Step II.
Notice that $(\Phi_n)\subset W^{-1,p'}(\Omega_0)$ and $(\mu_n)$ defines a sequence
of Radon measures on $\Omega_0$.
Taking into account the strict convexity and monotonicity of $\{t\mapsto j(s,t)\}$
and the growth conditions~\eqref{j3}-\eqref{j4}, we claim that the operators
$b,b_n$ satisfy the following properties:
\begin{align}
    \label{1}
    & (b_n(x,\xi)-b_n(x,\xi'))\cdot(\xi-\xi')\geq 0,
    \quad\text{a.e.~$x\in\Omega_0$, for all $\xi,\xi'\in\R^N$,} \\
        \label{2}
    & (b(x,\xi)-b(x,\xi'))\cdot(\xi-\xi')>0,
    \quad\text{a.e.~$x\in\Omega_0$ with $u^*(x)\leq h$, for all $\xi\neq\xi'$}, \\
        \label{3}
&   b_n(x,\cdot)\to b(x,\cdot)\,\,\,\text{as $n\to\infty$},
\quad\text{a.e.~$x\in\Omega_0$, uniformly over compact sets of $\R^N$}, \\
    \label{4}
&   b_n(x,Du_n) \quad\text{is bounded in $L^{p'}(\Omega_0,\R^N)$}, \\
    \label{5}
&   b_n(x,Du^*)\to b(x,Du^*)\,\,\,\text{as $n\to\infty$},
\quad\text{strongly in $L^{p'}(\Omega_0,\R^N)$},  \\
\label{5piu}
& \mu_n \rightharpoonup \mu\,\,\,\text{as $n\to\infty$},
\quad\text{weakly* in measure, for some Radon measure $\mu$},  \\
\label{6}
&  \Phi_n \to \Phi\,\,\text{as $n\to\infty$},
\quad\text{strongly in $W^{-1,p'}(\Omega_0)$}.
\end{align}
Properties~\eqref{1} and~\eqref{2} follow from the strict convexity of the map $\{\xi\mapsto j(s,|\xi|)\}$
and the definition of $H$.
Concerning~\eqref{3}, given $x\in\Omega_0$ and a compact $K\subset\R^N$, again
by the definition of $H$ and the continuity of $j_t,j_{st}$, for all $\xi\in K$
and all $n\geq 1$ large,
\begin{align*}
    & |b_n(x,\xi)- b(x,\xi)|  \leq \big|H\big(\frac{u_n(x)}{h}\big)-H\big(\frac{u^*(x)}{h}\big)\big||j_t(u_n(x),|\xi|)| \\
& +H\big(\frac{u^*(x)}{h}\big)|j_t(u_n(x),|\xi|)-j_t(u^*(x),|\xi|)| \\
& \leq \sup_{s\in [0,2h+2],\xi\in K} |j_t(s,|\xi|)|\big|H\big(\frac{u_n(x)}{h}\big)-H\big(\frac{u^*(x)}{h}\big)\big| \\
& +H\big(\frac{u^*(x)}{h}\big)|j_{st}(\tau u_n(x)+(1-\tau)u^*(x),|\xi|)||u_n(x)-u^*(x)| \\
& \leq \sup_{s\in [0,2h+2],\xi\in K} |j_t(s,|\xi|)|\times\big|H\big(\frac{u_n(x)}{h}\big)-H\big(\frac{u^*(x)}{h}\big)\big| \\
& +\sup_{s\in [0,2h+1],\xi\in K} |j_{st}(s,|\xi|)|\times |u_n(x)-u^*(x)| \\
& \leq C_{h,K}\Big\{\big|H\big(\frac{u_n(x)}{h}\big)-H\big(\frac{u^*(x)}{h}\big)\big|+|u_n(x)-u^*(x)|\Big\},
\end{align*}
which yields the assertion. Conclusion~\eqref{4} follows from the inequality
$$
|b_n(x,Du_n)|^{p'}=\Big|H\Big(\frac{u_n(x)}{h}\Big)j_t(u_n(x),|Du_n|)\Big|^{p'}\leq (\gamma (2h))^{p'}|Du_n|^p,
$$
for a.e.~$x\in\Omega_0$.
Moreover, since $b_n(x,Du^*(x))\to b(x,Du^*(x))$ a.e.~in $\Omega_0$ and
$$
|b_n(x,Du^*)|^{p'}=\Big|H\Big(\frac{u_n(x)}{h}\Big)j_t(u_n(x),|Du^*|)\Big|^{p'}\leq (\gamma (2h))^{p'}|Du^*|^p,
$$
for a.e.~$x\in\Omega_0$,
\eqref{5} holds as well by dominated convergence. Concerning~\eqref{5piu},
it can be easily verified that the square bracket in~\eqref{misradon}
is a bounded sequence in $L^1(\Omega_0)$ (just argue as in the estimation
    of the $I_i^n$s), so that, up to a subsequence, the property holds.
    Let us now prove that~\eqref{6} holds. In fact,
    for all $v\in W^{1,p}_0(\Omega_0)$ such that $\|v\|_{W^{1,p}(\Omega_0)}\leq 1$, we have
    \begin{align*}
    |\langle \Phi_n-\Phi,v \rangle| &\leq C
    \|f(|x|,u_n)H\big(\frac{u_n}{h}\big)-f(|x|,u^*)H\big(\frac{u^*}{h}\big)\|_{L^\frac{p^*}{p^*-1}(\Omega_0)} \\
    &+C\lambda_n\|g(u_n)H\big(\frac{u_n}{h}\big)-g(u^*)H\big(\frac{u^*}{h}\big)\|_{L^\frac{p^*}{p^*-1}(\Omega_0)} \\
    &+C|\lambda_n-\lambda|\|g(u^*)H\big(\frac{u^*}{h}\big)\|_{L^\frac{p^*}{p^*-1}(\Omega_0)}.
    \end{align*}
    Since $f(|x|,u_n)H\big(\frac{u_n}{h}\big)-f(|x|,u^*)H\big(\frac{u^*}{h}\big)\to 0$ and
    $g(u_n)H\big(\frac{u_n}{h}\big)-g(u^*)H\big(\frac{u^*}{h}\big)\to 0$ as $n\to\infty$, a.e.~in $\Omega_0$,
    $\lambda_n\to\lambda$ as $n\to\infty$ and, by the growth assumptions~\eqref{gg1} and~\eqref{gg1-G},
    \begin{align*}
    & |f(|x|,u_n)H\big(\frac{u_n}{h}\big)-f(|x|,u^*)H\big(\frac{u^*}{h}\big)|^\frac{p^*}{p^*-1}\leq C_h\, a^\frac{p^*}{p^*-1}(|x|)+C_h',
    \quad\text{a.e.~in $\Omega_0$},\\
    & |g(u_n)H\big(\frac{u_n}{h}\big)-g(u^*)H\big(\frac{u^*}{h}\big)|^\frac{p^*}{p^*-1}\leq C_h{''},
    \quad\text{a.e.~in $\Omega_0$},
\end{align*}
    for some $C_h,C_h',C_h^{''}>0$, we obtain the
    property by taking the supremum on $v$ and dominated convergence.
    Therefore, in light of~\eqref{weakustarbdd} and~\eqref{1}-\eqref{6},
    we can apply~\cite[Theorem 5]{masmur}, yielding the almost everywhere
    convergence of the gradients $Du_n$ to $Du^*$ on the set
    $$
    E_{h,\Omega_0}=\{x\in \Omega_0: u^*(x)\leq h\}.
    $$
    We deduce the pointwise convergence~\eqref{convun2} by the arbitrariness of $h\geq 1$ and $\Omega_0\subset\Omega$.
\vskip4pt
\noindent
{\bf Step IV (proof of the theorem concluded).} In view of~\eqref{convun2} and~\eqref{j2}, we have
\begin{align*}
    & j(u_n,|D u_n|)-\alpha_0 |D u_n|^p\geq 0,\qquad\text{for all $n\geq 1$ and a.e.~in $\Omega$},  \\
    & j(u_n,|D u_n|)-\alpha_0 |D u_n|^p\to j(u^*,|D u^*|)-\alpha_0 |D u^*|^p,\quad\text{as $n\to\infty$, a.e.~in $\Omega$}.
\end{align*}
    Taking into account~\eqref{energyconv-final}, by Fatou's lemma, we get
    $$
    \limsup_n\int_\Omega |D u_n|^p\leq \int_\Omega |D u^*|^p.
    $$
    Since $(u_n)$ converges to $u^*$ strongly in $L^p(\Omega)$
    and weakly in $W^{1,p}_0(\Omega)$, we can conclude that $u_n\to u^*$ strongly in $W^{1,p}_0(\Omega)$, as $n\to\infty$.
    Taking the limit into~\eqref{gradientpp} we reach
    \begin{equation*}
    \|Du\|_{L^p(\Omega)}=\|Du^*\|_{L^p(\Omega)}.
\end{equation*}
    Then, by~Proposition~\ref{identityc}, there is a translate of $u^*$ which is almost
    everywhere equal to $u$.
\qed

\vskip10pt
\noindent
{\bf Acknowledgment.} The second author thanks Marco Degiovanni for a helpful discussion
about the proof of Lemma~\ref{subdcaract}.

\bigskip

\end{document}